\documentclass[11pt]{amsart}

\newtheorem{theorem}{Theorem}[section]

\theoremstyle{definition}

\theoremstyle{remark}

\numberwithin{equation}{section}

\newcommand\lam{\lambda}

\newcommand\Del{\Delta}

\newcommand\ppt{\frac{\partial}{\partial t}}
\newcommand\ddt{\frac{d}{dt}}
\newcommand\vol{\textup{\textsf{Vol}}}

\begin{document}

\title[A Comparison Theorem]{A Comparison Theorem via the Ricci Flow}

\author[Jun Ling]{Jun \underline{LING}}
\address{Department of Mathematics, Utah Valley University, Mail stop 261, 800 W. University Pkwy, Orem, Utah 84058, USA}
\email{lingju@uvu.edu}
\thanks{The author thank the Mathematical Sciences Research Institute at Berkeley for its
hospitality for program `The Geometric Evolution Equations and
Related Topics' and thank National Science Foundation for the
support offered.}



\date{October 8, 2007, Revised: Aril 6, 2009}



\begin{abstract}
We prove a comparison theorem for the compact surfaces
with negative Euler characteristic
via the Ricci flow.
\end{abstract}

\maketitle

\section{Introduction}\label{sec-intro}

A good way to understand geometry of manifolds is to compare
general manifolds with the ones with constant curvature. There
have been many such comparison results. Among them are Bishop and
Guenther's volume comparison theorems, Cheeger-Yau's heat kernel
comparison theorem, Faber-Krahn's comparison theorem, to name only
a few (see \cite{cha}, \cite{schoen-yau}). Those comparison
theorems not only have great impact on geometry but also are
particularly beautiful. For example, Faber-Krahn's theorem (see
\cite{cha}, \cite{schoen-yau}) states that in Euclidean space
$\mathbb{R}^n$, the first Dirichlet eigenvalue of the Laplacian
of a domain $\Omega$ is greater than or equal to that of a ball
$B$ as long as $\Omega$ and $B$ have the equal volume. In this
paper, we prove a new comparison theorem (Theorem \ref{thm1})
that may be considered as an analogue for closed surfaces.
While Faber-Krahn's is on domains of flat space $\mathbb{R}^n$,
ours is on curved surfaces. One can expect that the curvature enters to play,
as showed in Theorem \ref{thm1}. Theorem \ref{thm1} actually implies
a sharp upper bound for the $i^{\textup{th}}$ eigenvalue. That is stated in
Theorem \ref{thm2}.

We prove our results by using Hamilton's Ricci flow. The proofs
are enlightened by Perelman's work \cite{perelman} on the
monotonic property of the first eigenvalue of the operator
$-4\Delta +R$ under the Ricci flow. Perelman \cite{perelman} used
that and nondecreasing property of his entropy under the Ricci
flow to rule out nontrivial steady or expanding breathers on
compact manifolds, among other applications.

In this paper, we let $M$ be a compact surface without boundary.
For any Riemannian metric $g$ on $M$, we let $K_g$ be the Gauss
curvature, $\kappa_g$ the minimum of the Gauss curvature,
$\vol_g(M)$ the volume of $M$, $d\mu_{g}$ the volume element,
$\Delta_g$ the Lalacian of the metric $g$. We have
the following results.

\begin{theorem}[Comparison Theorem]\label{thm1}
Let $M$ be a compact surface with Euler Characteristic $\chi<0$,
$g$ any Riemannian metric on $M$, and $\tilde{g}$ a Riemannian
metric on $M$ that has constant Gauss curvature $K_{\tilde{g}}$ and that is in the
same conformal class as g. If\, $\vol_g(M)=\vol_{\tilde{g}}(M)$, and if
the $i^{\textup{th}}$
eigenvalue $\lam_g$ of the Lapacian $\Delta_g$ has $\textup{C}^1$-dependence on metric $g$,
then we have
\[
\frac{\lam_g}{\kappa_g}\geq
\frac{\lambda_{\tilde{g}}}{\kappa_{\tilde{g}}},
\]
where the constant Gauss curvature for metric $\tilde{g}$ is
$K_{\tilde{g}}=2\pi\chi/\vol_g(M)$.
\end{theorem}

\begin{theorem}[Sharp Upper Bound]\label{thm2}
Let $M$ be a compact surface with Euler Characteristic $\chi<0$,
$g$ any Riemannian metric on $M$, and $\tilde{g}$ a Riemannian
metric on $M$ that has constant Gauss curvature $K_{\tilde{g}}$ and that is in the
same conformal class as g.  Let $\sigma$ be a lower bound of the
Gauss curvature $K_g$. If\, $\vol_g(M)=\vol_{\tilde{g}}(M)$, and if
the $i^{\textup{th}}$
eigenvalue $\lam_g$ of the Lapacian $\Delta_g$ has $\textup{C}^1$-dependence on metric $g$,
then $\lam_g$ has an upper bound
\[
\lambda_g\leq\frac{\lambda_{\tilde{g}}}{\kappa_{\tilde{g}}}\,
\,\sigma,
\]
that is,
\[
\lambda_g\leq\frac{\lambda_{\tilde{g}}}{2\pi\chi}\,
\vol_{\tilde{g}}(M)\,\sigma,
\]
where the constant Gauss curvature for metric $\tilde{g}$ is
$K_{\tilde{g}}=2\pi\chi/\vol_g(M)$.

In particular, we have
\[
\lambda_g\leq\frac{\lambda_{\tilde{g}}}{\kappa_{\tilde{g}}}\,\kappa_g,
\]
that is
\[
\lambda_g\leq\frac{\lambda_{\tilde{g}}}{2\pi\chi}\,
\vol_{\tilde{g}}(M)\,\kappa_g,
\]
for which, the equality holds for metric $g$ with constant Gauss
curvature.
\end{theorem}

\section{Proofs}\label{sec-proofs}
We prove Theorem \ref{thm2} first, and prove Theorem \ref{thm1}
after that.
\begin{proof}[Proof of Theorem \ref{thm2}]
We evolve metric $g$ in the theorem by the normalized Ricci flow.
Let $g(t)$ be the solution of the normalized Ricci flow
\begin{equation}\label{nrf}
\ppt g(t)=(r-R)g(t),
\end{equation}
with initial value
\begin{equation}\label{ic}
g(0)=g,
\end{equation}
where $R$ is the the scalar curvature of the metric $g(t)$, $r$
the average of the scalar curvature
\[
r=\int_M Rd\mu\Big/\int_Md\mu,
\]
and $d\mu=d\mu_t$ the volume element of $g(t)$. It is known from
\cite{hamilton}, \cite{hamilton2}, see also \cite{chow-knopf},
that $g(t)$ exists for all $t\geq 0$.

By (\ref{nrf}) we have
\[
\frac{d}{dt}(d\mu) =(r-R)d\mu
\]
and
\[
\frac{d}{dt}\,\vol_{g(t)}(M)=\frac{d}{dt}\int_M 1
d\mu=\int_M(r-R)d\mu=0.
\]
So the volume $\vol_{g(t)}(M)$ remains constant in $t$, that is
\[
V=:\vol_{g(t)}(M)=\vol_{g(0)}(M) \quad \forall t\geq 0.
\]
By the Gauss-Bonnet Theorem, we have
\begin{equation}\label{r-chi}
r=4\pi\chi/V<0.
\end{equation}
Therefore $r$ is a negative constant and the lower bounds of $R$
are negative as well. It is known from \cite{hamilton2}, see also
\cite{chow-knopf}, that $g(t)$ converges exponentially in any
$C^k$-norm to a smooth metric $g(\infty)$ that has constant Gauss
curvature $r/2$.

$R/2$ is the Gauss curvature $K$ of the metric $g(t)$. Let
$\sigma<0$ be a lower bound of $K|_{t=0}=R|_{t=0}/2$,
\begin{equation}\label{kappa}
K|_{t=0}\geq \sigma.
\end{equation}

Let $\lambda$ be the $i^{\textup{th}}$ eigenvalue of the Laplacian
$\Delta$ of metric $g(t)$, $u$ the corresponding eigenfunction. Then we have
\[
-\Del u =\lambda u.
\]
Take derivatives with respect to $t$,
\[
-(\ppt \Del) u -\Del \ppt u =(\ddt\lambda) u+ \lambda\ppt u.
\]
Multiply the equation by $u$ and integrate,
\begin{equation}\label{ddt-eigen-eq}
-\int u(\ppt \Del) ud\mu  -\int u\Del \ppt ud\mu = (\ddt\lambda)
\int u^2d\mu+ \lambda\int u\ppt ud\mu.
\end{equation}
Standard calculations (Ch. \cite{chow-knopf}) show that
\begin{equation}\label{ddt-laplacian}
\ppt(\Del_{g(t)})=(R-r)\Del_{g(t)}.
\end{equation}

Now (\ref{ddt-eigen-eq}), (\ref{ddt-laplacian}) and the equation
\[
-\int u\Del \ppt u=-\int\Del u\, \ppt u=\lambda\int u\ppt u,
\]
imply that
\[
\begin{split}
\frac{d}{dt}\lambda&=-\int u(\ppt \Del) ud\mu\Big/\int u^2d\mu\\
&=-\int u(R-r)\Del ud\mu\Big/\int u^2d\mu
\\
&=\lambda\int (R-r\big)u^2d\mu\Big/\int u^2d\mu.\\
\end{split}
\]
Therefore we have
\[
\lambda(t)=\lambda(0)\exp\Big\{\int_0^t\frac{\int_M (R-r)u^2d\mu}{\int_M u^2d\mu}
d\tau\Big\}.
\]
The scalar curvature $R$ evolves by the equation
\[
\ppt R=\Del R+ R(R-r).
\]
Using the maximum principle to compare $R$ with the solution
\[
s(t)=\frac{r}{1-\left(1-\frac{r}{2\sigma}\right)e^{rt}}
\]
of the initial value problem of ODE
\[
\left\{
\begin{split}
&\frac{d}{dt}s=s(s-r),\\
&s(0)=2\sigma,
\end{split}
\right.
\]
we get
\[
R\geq \frac{r}{1-\left(1-\frac{r}{2\sigma}\right)e^{rt}}\qquad
\forall x\in M, \forall t\in[0,\infty),
\]
and
\[
\begin{split}
&\int_0^t\frac{\int_M\Big(R(\tau)-r\Big)u^2d\mu}{\int_M u^2d\mu} d\tau \\
&\geq
\int_0^t\frac{\int_M\left(\frac{r}{1-\left(1-\frac{r}{2\sigma}\right)e^{r\tau}}-r\right)u^2d\mu}{\int_M u^2d\mu}
d\tau
\\
&=\ln\frac{\frac{r}{2\sigma}e^{rt}}{1-\left(1-\frac{r}{2\sigma}\right)e^{rt}}-rt.
\end{split}
\]
Therefore
\[
\lambda(t)\geq\lambda(0)\frac{\frac{r}{2\sigma}e^{rt}}{1-\left(1-\frac{r}{2\sigma}\right)e^{rt}}\Big/e^{rt}
=\lambda(0)\frac{\frac{r}{2\sigma}}{1-\left(1-\frac{r}{2\sigma}\right)e^{rt}}.
\]
Note that the eigenvalues of the Laplacian continuously depends on the metric. 
Letting $t\rightarrow \infty$,  we have
\begin{equation}\label{sharp-bound}
\lambda(\infty)\geq\lambda(0)\frac{r}{2\sigma},
\end{equation}
where $\lam(\infty)$ is the $i^{\textup{th}}$ eigenvalue of the
Laplacian of the metric $g(\infty)$ that has constant Gauss
curvature $r/2$.

Let $\tilde{g}=g(\infty)$. It is clear that
$\kappa_{\tilde{g}}=r/2$, since Gauss curvature $K_{\tilde{g}}$
for metric $\tilde{g}$ is constant. Taking these with (\ref{ic})
into (\ref{sharp-bound}), we get the first estimate in the
theorem, which with (\ref{r-chi}) yield the second estimate in the
theorem
\[
\lambda_{g}\leq\frac{\lambda_{\tilde{g}}}{2\pi\chi}\,\sigma
\vol_{\tilde{g}}(M).
\]

By the Gauss-Bonnet Theorem, the constant Gauss curvature for
metric $\tilde{g}$ is
$r/2=2\pi\chi/\vol_{\tilde{g}}(M)=2\pi\chi/\vol_g(M)$.

Letting $\sigma=\kappa_g$, we get the third and fourth estimates
in the theorem from the first two.
\end{proof}

\begin{proof}[Proof of Theorem \ref{thm1}]

Letting $g(\infty)=\tilde{g}$, $\sigma=\kappa_g$ and
$r/2=\kappa_{\tilde{g}}$ in (\ref{sharp-bound}), with (\ref{ic}),
we get
\[
\lambda_g\leq\lambda_{\tilde{g}}\frac{\kappa_g}{\kappa_{\tilde{g}}}.
\]
Therefore we have
\[
\frac{\lambda_g}{\kappa_g}\geq\frac{\lambda_{\tilde{g}}}{\kappa_{\tilde{g}}}
\]
The last estimate follows from the above one.
\end{proof}

\begin{center}
\textbf{Acknowledgement}
\end{center}
The author is grateful to Professor Yieh-Hei Wan for his constant
support, valuable suggestions and helps,  to Professor Bennett
Chow for his many enlightening talks, papers and books,
suggestions and helps.

\end{document}